\documentclass[12pt,letterpaper]{amsart}
\usepackage{amsmath} 
\usepackage{graphicx} 
\usepackage{latexsym, amsmath, amssymb, amscd, amsfonts, url, color}

\usepackage{fullpage}

\usepackage{xy}
\usepackage{tikz}
\newcommand{\op}{\operatorname}
\newcommand{\R}{\mathbb{R}}
\newcommand{\Z}{\mathbb{Z}}
\newcommand{\stab}{\operatorname{stab}}
\newcommand{\x}{\chi}
\newcommand{\wt}{\operatorname{wt}}
\newcommand{\infl}{\operatorname{inf}}
\newcommand{\C}{\mathbb{C}}

\newcommand{\old}[1]{{}}
\newcommand{\CP}{\mathbb{C}\mathbb{P}}

\theoremstyle{remark}
\newtheorem{example}[equation]{Example}
\newtheorem{counterexample}[equation]{Counterexample}
\newtheorem{remark}[equation]{Remark}

\theoremstyle{plain}
\newtheorem{theorem}[equation]{Theorem}
\newtheorem{lemma}[equation]{Lemma}

\newtheorem{corollary}[equation]{Corollary}

\newtheorem{defn}[equation]{Definition}

\definecolor{forestgreen}{rgb}{0.0, 0.27, 0.13}

\newcommand{\hidden}[1]{\footnote{Hidden:  #1}}
\renewcommand{\hidden}[1]{}

\begin{document}
\title{An enriched count of nodal orbits in an invariant pencil of conics}

\author[Bethea]{Candace Bethea}

\begin{abstract}

 This work gives an equivariantly enriched count of nodal orbits in a general pencil of plane conics that is invariant under a linear action of a finite group on $\CP^2$. 
 This is both inspired by and a departure from $R(G)$-valued enrichments such as Roberts's equivariant Milnor number and Damon's equivariant signature formula. Given a $G$-invariant general pencil of conics, the weighted sum of nodal orbits in the pencil is a formula in $A(G)$ in terms of the base locus considered as a $G$-set. We show this is true for all finite groups except $\mathbb{Z}/2\times \mathbb{Z}/2$, $A_4$, and $D_8$ and give counterexamples for the  exceptional groups. 
\end{abstract}

\maketitle

\section{Introduction}

Given a pair of conics defined by equations $f$ and $g$ which intersect generically in $\CP^2$, there is a family of curves parameterized by $\CP^1$ $$X=\{\mu f(x,y,z)+\lambda g(x,y,z)=0\colon [\mu: \lambda] \in\CP^1\}.$$  This is the pencil of conics spanned by $f$ and $g$, and specifying different $[\mu:\lambda]$ in $\CP^1$ specifies different conics in the pencil. The set $$\Sigma=\{p\in\CP^2\colon f(p)=g(p)=0\}$$ is called the base locus of $X$, and it contains 4 distinct points since $f$ and $g$ intersect generically. It is natural to ask how many conics in $X$ are nodal. As long as $f$ and $g$ are in general position, there are always  $|\Sigma|-1=4-1=3$ nodal conics in $X$. This is a classical result that is also case of G\"{o}ttsche's conjecture \cite{Got98}, which was proved by Y. Tzeng in \cite{Tzeng12} and  Kool-Schende-Thomas in \cite{KST11}. 

Rather than ask for the number of nodal conics, one can ask if there is a description of orbits of  nodal conics under the presence of a finite group action on $\CP^2$ under which the pencil is invariant. This work gives a formula for the count of nodal orbits when 
the group $G$ acting on $\CP^2$ is not isomorphic to $\Z/2\times \Z/2$, $A_4$, or $D_8$. We define the $G$-weight of a nodal orbit, which could be considered a special case of a Burnside-valued Milnor number inspired by the $R(G)$-valued Milnor number of \cite{Roberts85} and signature formula of \cite{Damon91}. The main result of this work shows that the sum of these $G$-weights of nodal orbits is equal to $[\Sigma]-\{*\}$ in the Burnside Ring, $A(G)$. 

To be precise, let $G$ be a finite group that acts linearly on $\CP^2$, and let $f$ and $g$ be equations defining a general pair of conics in $\CP^2$ such that the corresponding pencil, $X$, is $G$-invariant.  
From the linear action on $\CP^2$, we obtain an action on $\text{Sym}^2(\mathbb{C}^3)^{\vee} = \text{Span}\{x^2,y^2,z^2,xy,xz,yz\}$. 
 $X$ being $G$-invariant means that $g\cdot C_t$ is another conic in $X$ for all $g$ in $G$ and for all conics $C_t$ in $X$, where we've written $t=[\mu: \lambda]\in \CP^1$ for simplicity so that $C_t$ is the conic obtained by specifying $t$. In other words, $X$ is $G$-invariant if $\langle f, g\rangle$ is a $G$-invariant subspace of $\text{Sym}^2(\mathbb{C}^3)^{\vee}$. 
With this setup, one can ask for a Burnside-valued formula of $G$-sets counting orbits of nodal conics in the $G$-invariant pencil of conics $X$. Given such a formula, we can take the cardinality of the $H$-fixed points of the formula for any subgroup $H$ of $G$ to obtain an integer count of nodal conics that are fixed by $H$. 

The main theorem is proved in Theorem \ref{MAIN} in Section \ref{mainsection}, stated below: 

\begin{theorem}\label{thm:intro_main_thm} Let $G$ be a finite group not isomorphic to  $\mathbb{Z}/2\times\mathbb{Z}/2$, $A_4$, or $D_8$, and assume $G$ acts linearly on $\CP^2$. Let $X$ be a $G$-invariant general pencil of conics in $\CP^2$, and let $[\Sigma]$ in $A(G)$ represent the base locus of $X$. Then 
\begin{equation}\label{eq:intro_main_eq}\sum_{ \substack{G\cdot C_t\text{, }  \\ C_t \in X \text{ is nodal}}} \wt^G(C_t)=[\Sigma]-\{*\}\end{equation}
 in $A(G)$. That is, there is a weighted count of orbits of nodal conics in $X$, valued in the Burnside ring of $G$.
\end{theorem}
For any subgroup $H$ of $G$, the cardinality of $$([\Sigma]-\{*\})^H
$$ is equal to the number of nodal conics in $X$ that are fixed by $H$.  In particular, we recover the classical count of $|\Sigma|-1=3$ nodal conics by taking $H$ to be the trivial subgroup, and we recover the number of nodal conics that are fixed under the action on $\CP^2$ by taking $G$-fixed points. In this sense, the equivariant enrichment in Theorem \ref{thm:intro_main_thm} is a  generalization of the classical result counting $|\Sigma|-1$ nodal conics in a general pencil. This is a virtual count, in particular there may be cases where the cardinality of the fixed points of \eqref{eq:intro_main_eq} is equal to $-1$,  so it is crucial that Theorem \ref{thm:intro_main_thm}  be interpreted as a virtual weighted invariant count, with the weight of a nodal conic defined in Definition \ref{weight}. 

The Burnside Ring, $A(G)$, of a finite group $G$ is the Grothendieck group completion the monoid of $G$-isomorphism classes of finite $G$-sets with addition given by disjoint union. $A(G)$ also has a ring structure given by Cartesian product. 
Equivariant formulas of $G$-sets should be valued in $A(G)$, as above, as $A(G)$ distinguishes equivariant homotopy classes of endomorphisms of $G$-representation spheres. 
Specifically,  $$\text{deg}^G\colon \left[S^V, S^V\right]^G\stackrel{\sim}{\longrightarrow} A(G)$$ is an isomorphism 
 (see \cite{S70}), analogous to $\text{deg}\colon \left[S^n,S^n\right]\stackrel{\sim}{\longrightarrow} \mathbb{Z}$ being an isomorphism,  which motivates the replacement of $\mathbb{Z}$ by $A(G)$ as the ring of definition for equivariant enumerative results. Further description of the Burnside ring is given in Section \ref{Notationsection}.

The weighting convention for nodal orbits in $X$ appearing in the left-hand side of equation \eqref{eq:intro_main_eq} is defined in Section \ref{mainsection} before the main theorem is restated, and it generalizes the real sign of a node in the sense that the $G$-fixed point cardinality of the weight of a real node is $+1$ or $-1$ depending whether the node is split or non-split when $G=\op{Gal}(\C/\R)$ acts on $\CP^2$. 

\noindent{\bf Acknowledgements.} The author would like to thank Jesse Kass, Kirsten Wickelgren, and Brendan Hassett for being fantastic PhD and postdoctoral advisers. I would also like to thank Sabrina Pauli, Thomas Brazelton, and Alicia Lamarche for helpful conversations that have enhanced my understanding of enumerative geometry  and enrichments of enumerative results. I was supported by an SREB Dissertation Fellowship while conducting this work. I am currently supported by NSF DMS-2402099.

\section{Notation and definitions}\label{Notationsection}

This section will introduce all definitions related to the Burnside Ring so the paper is self-contained, following \cite{TD79}. We will always assume $G$ is a finite group, and all group actions are assumed to be left actions. Given $G$-sets $S$ and $T$, a set map $f\colon S\to T$ is $G$-equivariant if $g\cdot f(s)=f(g\cdot s)$ for all $g$ in $G$. 
Given a $G$-set $S$ and a subset $S'$ of $S$, we will say $S'$ is $G$-invariant  
if $g\cdot s'$ is in $S'$ for all $s'$ in $S'$ and $g$ in $G$.

Given any two $G$-sets $S$ and $T$, there are natural set operations on $S$ and $T$ that form new $G$-sets. We can take the disjoint union of $S$ and $T$, $S\sqcup T$, or the Cartesian product, $S\times T$, and obtain $G$-sets by letting $G$ act diagonally in both cases. Let $A(G)^{+}$ denote the semi-ring of $G$-isomorphism classes of finite $G$-sets with addition given by disjoint union and multiplication by Cartesian product, where $G$-isomorphism means an isomorphism of sets that is $G$-equivariant. 

\begin{defn} Given a group $G$, the \emph{Burnside ring} of $G$, denoted $A(G)$, is the Grothendieck group completion of $A(G)^+$, which is also a ring under Cartesian product. 
\end{defn}

Additively, $A(G)$ is the free abelian group on isomorphism classes of transitive $G$-sets of the form $[G/H]$ for subgroups $H$ of $G$. Given any $G$-set $S$, we will denote its class in $A(G)$ by $[S]$. The set $\{*\}$ will always denote the one-point set, also commonly denoted $[G/G]$ in terms of additive generators. 
Any $G$-set which comes from a genuine set with a group action will be called a genuine $G$-set. This is in contrast to a virtual $G$-set, 
for example, $-\{*\}$ denotes the virtual $G$-set that is the formal additive inverse of the  $G$-set $\{*\}$. Literature on the Burnside ring is rich, a standard reference being \cite{TD79}. 

Any genuine finite $G$-set can be written as the disjoint union of its orbits $\sqcup G/H_i$ where $\{H_i\}$ is some finite collection of subgroups of $G$.  
Furthermore, the isomorphism class of $[G/H]$ in $A(G)$ is determined by the conjugacy class of $H$ in $G$. Thus every genuine $G$-set $[S]$ in $A(G)$ can be written as $$[S]=\sum_{H_i\leq G}n_i[G/H_i]$$ for some positive integers $n_i$, uniquely up to choice of conjugacy class representative for each $H_i$. 
A well-known result in equivariant topology, see Proposition 1.2.2 in  \cite{TD79}, says that two isomorphism classes of finite $G$-sets, $[S_1]$ and $[S_2]$, are equal in $A(G)$ if $|(S_1)^H|=|(S_2)^H|$ for all subgroups $H$ of $G$. These facts are simple to state but useful in practice, and will be important in Section \ref{counterexamplesection}.

Given $G$-sets $S$ and $T$, the group and ring operations in $A(G)$ produce new $G$-sets given by disjoint union and Cartesian product. Given a finite group $G$ and a subgroup $H$ of $G$, we use the inflation from $H$ to $G$ as a change of group method to obtain a $G$-set from an $H$-set.  
\begin{defn}  Given a subgroup $H$ of $G$ and an $H$-set X, we define a $G$-set with underlying set structure $(G\times X)/\sim$, where $(gh,x)\sim(g,hx)$ for all $h$ in $H$, and $x$ in $X$. The {\em inflation of $X$ from $H$ to $G$}, denoted $\infl_{H}^G(X)$, is the $G$-set $(G\times X)/\sim$ with $G$-action given by $g'\cdot(g,x)=(g'g,x)$ for all $g'\in G$ and $(g,x)\in\infl_H^G(X)$. \end{defn}

Every genuine $G$-set we will encounter in this paper will already be represented as a formal sum of orbits, each equal to $[G/H]$ in $A(G)$ for some subgroup $H$ of $G$. Thus it will be useful to have a description of the inflation of an $H$-set to $G$ when represented by a sum of orbits of this form. The following lemma gives such a description.

\begin{lemma}\label{inflation} Let $H$ be a subgroup of a finite group $G$ and let $[X]$ be a finite $H$-set the form $$[X]=\sum_{i=1}^m n_i[H/K_i]$$ in $A(H)$ for some $m\in \mathbb{N}$, some $n_i\in \mathbb{Z}$, and $K_i\leq H$ some finite collection of subgroups of $H$. Then $\infl_H^G(X)=\sum_{i=1}^m n_i[G/K_i]$ in $A(G)$.
\end{lemma}

\begin{proof} First note  $$\infl_H^G(\sum_{i=1}^m n_i[H/K_i])=\sum_{i=1}^m n_i\infl_H^G([H/K_i])$$ because the action on a disjoint union is assumed to be diagonal. Thus we only need to show that $\infl_H^G(H/K)=[G/K]$ in $A(G)$ for any $K\leq H$, i.e., by defining a set isomorphism $\infl_H^G(H/K)\to G/K$ and showing it is $G$-equivariant.

Define $f\colon \infl_H^G(H/K)\to G/K$ by $f((g,hK))=ghK$. 
$f$ is well-defined, injective, and surjective as a set function. 
Finally, $f$ is $G$-equivariant, as
  $$g'\cdot f((g,hK))= 
g'ghK=f((g'g,hK))=f(g'\cdot(g,hK))$$
 for all $g'$ in  $G$ and $(g,hK)$ in $\infl_H^G(H/K)$. 
\end{proof}

\section{Proof of the Classical Result}\label{classicalproof}

We sketch a topological proof of the classical result, which is well-known to enumerative geometers but may be new for some. 
Let $f$ and $g$ be a general pair of conics in $\CP^2$, and let $$X=\{\mu f+\lambda g=0\colon [\mu, \lambda] \in\CP^1\}\subseteq\CP^2$$ be the pencil of conics defined by $f$ and $g$.  
Let $$X_{tot}=\{([\mu:\lambda],p)\colon \mu f(p)+\lambda g(p)=0\}\subseteq \CP^1\times \CP^2$$ be the total space of $X$. We have two projections from $X_{tot}$, $\pi_1\colon X_{tot}\to \CP^1$ by projecting onto the first coordinate and $\pi_2\colon X_{tot}\to \CP^2$ by projecting onto the second coordinate. We will compute $\chi(X_{tot})$ in two ways and set them equal to obtain the number of nodal conics.

First we will compute $\chi(X_{tot})$ using the projection $\pi_1\colon X_{tot}\to \CP^1$. Let $$D=\{[\mu, \lambda]\colon \mu f+\lambda g=0\}\subseteq\CP^1$$ be the set of points in $\CP^1$ that specify a nodal conic in $X$. Note that $|D|$ is equal to the number of nodal conics in $X$, which is what we want to find. The fibers of $\pi_1$ over $D$ are nodal conics, and the fibers over $\CP^1-D$ are smooth conics.
The compactly supported Euler characteristic is multiplicative in fiber bundles and additive over the base $\CP^1 = (\CP^1-D)\sqcup D$, thus
\begin{align*}
\chi(X_{tot})&=\chi(X_{tot}|D)+\chi(X_{tot}|\CP^1-D)\\
&=\x(C_{sing})\cdot\x(D)+\x(C_{sm})\cdot\x(\CP^1-D)
\end{align*}
where $C_{sm}$ denotes any smooth conic in a fiber over $\CP^1-D$ and $C_{sing}$ denotes any nodal conic in a fiber over $D$. This is the the topological Hurwitz formula applied to $X_{tot}\to \CP^1$,  if $E\to B$ is a fiber bundle with fiber $F$ and $B$ is path connected then 
$\x(F)\cdot\x(B)=\x(E)$.

The Euler characteristic of a smooth genus $g$ plane curve is $2-2g$, and the Euler characteristic of a genus $g$ plane curve $C$ with an isolated singularity at a point $p$ is $2-2g+\mu_p(C)$ where $\mu_p(C)$ is the Milnor number of $C$ at $p$. Since conics have genus 0 and the Milnor number of an ordinary nodal singularity is 1, $\chi(C_{sing}) = 3$ and 
\begin{align*}
\chi(X_{tot})&=\x(C_{sing})\cdot\x(D)+\x(C_{sm})\cdot\x(\CP^1-D)\\
&=3|D|+2(2-|D|)\\
&=|D|+4.
\end{align*}
 
The second way to compute $\x(X_{tot})$ is to use the fact that $\pi_2\colon X_{tot}\to \CP^2$ is the blow-up of $\CP^2$ at the $d^2=4$ points of the base locus $$\Sigma=\{p\in\CP^2\colon f(p)=g(p)=0\},$$ where $d=2$ is the degree of $f$ and $g$ as homogenous polynomials in three variables.   
Again using additivity for the compactly supported Euler characteristic, we have
\begin{align*}
\x(X_{tot})
&= \x(\CP^2)+ \x(\Sigma)(\x(\CP^1)-\x(pt))\\
&= 3+4(2-1)\\
&= 7. 
\end{align*}

Combining the two calculations of $\chi(X_{tot})$ we get  \[|D|+4=7,\] and we conclude that the number of nodal conics in $X$ is $|D|=3$. This approach works equally well for higher degree degree curves intersecting generically. Another proof can obtained by taking the degree of the top Chern class of the bundle of principle parts on $\mathcal{O}(d)$, both approaches can be found in detail in \cite[Chapter 7]{EH16}. As mentioned in the introduction, it is worth noting that another way to write the formula for $|D|$ is \[|D|=|\Sigma|-1,\] which motivates the form of Equation \eqref{eq:intro_main_eq} in Theorem \ref{thm:intro_main_thm}.

An elementary geometric proof of the same result can be described as follows. If $f$ and $g$ define a general pair of conics, then they intersect in exactly the four points of $\Sigma$. A nodal conic geometrically has irreducible components equal to a pair of lines, and the generic intersection assumption on $f$ and $g$ rules out the possibility that the two curves share a common line. Thus, asking how many conics in $X$ are nodal is equivalent to asking how many ways there are to draw a pair of distinct lines through four points in $\CP^2$, which is three.  
Labeling the points of $\Sigma$ as $b_1, b_2, b_3$, and $b_4$, 
the three pairs of lines are: 

\begin{figure}[ht]
\centerline{\includegraphics[width=.75\textwidth]{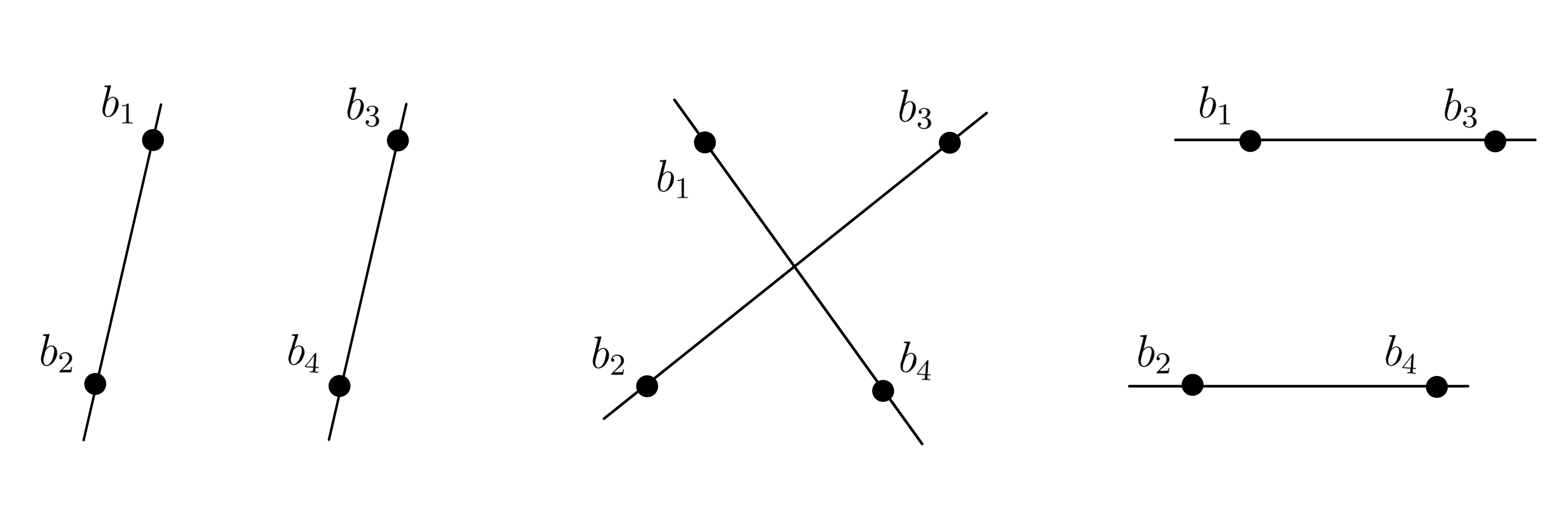}}
\caption{Distinct lines through $\Sigma=\{b_1, b_2, b_3, b_4\}$}
\label{figure:lines_through_Sigma}
\end{figure}

When describing the $G$-orbits of nodal conics, we can instead describe orbits of configurations of disjoint lines through $\Sigma$. 
 
Finally, note that any set of four points in $\mathbb{CP}^2$ uniquely determines a pencil of conics. Indeed, the vector space of conics in $\CP^2$ is 5-dimensional, $\text{Span}_{\mathbb{C}}\{x^2, y^2, z^2, xy, xz, yz\}$. Requiring that a conic passes through a point imposes a 1-dimensional condition on the space of conics, so requiring that a conic passes through four points results in a 1-dimensional linear span of conics, 
i.e., a pencil of conics. These conditions are linearly independent when the 4 points are in general position. Thus any $\Sigma$ which is a set of four points in general position $\CP^2$ uniquely determines a general pencil of conics.

\section{An Equivariant Count of Orbits of Nodal Conics} \label{mainsection}

Let $f$ and $g$ be homogeneous degree 2 equations defining a pair of plane conic curves that intersect generically, and let $$X=\{\mu f+\lambda g=0\colon [\mu:\lambda]\in\CP^1\}\subseteq \CP^2$$ be the corresponding pencil. We will henceforth write $t=[\mu:\lambda]$ so $C_t\in X$ denotes the element of $X$ obtained by specifying $[\mu: \lambda]$ in $\CP^1$.  
Given a nodal conic $C_t$ in $X$, we will write $B_t$ to denote the set of irreducible components of $C_t$.  Thus $B_t=\{L_1,L_2\}$ is the set of  branches of $C_t$ if $C_t$ is a nodal conic parameterized as a product of lines $L_1\times L_2$. 

\begin{defn}\label{weight} Let $C_t$ be a nodal conic in a $G$-invariant pencil of conics, and let $G_t\leq G$ be the stabilizer of $C_t$. Define the $G_t$-weight of $C_t$ to be $$\wt^{G_t}(C_t):=[B_t]-\{*\}$$ in $A(G_t)$ where $[B_t]$ denotes the branches of $C_t$ as a $G_t$-set in $A(G_t)\leq A(G)$. The \emph{$G$-equivariant weight} of the orbit  $G\cdot C_t$ is defined to be $$\wt^G(C_t):=\infl_{G_t}^G(\wt^{G_t}(C_t))$$ in $A(G)$.  
\end{defn} 

Given nodal conics $C_{s}$ and $C_{t}$ in the same orbit it is straightforward to check that $\wt^G(C_t)=\wt^G(C_s)$ in $A(G)$ using Lemma \ref{inflation}, so the weight of an orbit is well-defined. 

When the action of $G$ on $\CP^2$ is trivial,  $\stab(C_t)=G$ for all nodal $C_t$ in $X$, and so $\wt^G(C_t)=[B_t]-\{*\}$. Since the action is trivial, the branches in $[B_t]$ are fixed and $[B_t]=2\{*\}$ in $A(G)$. Thus $\wt^G(C_t)=2\{*\}-\{*\}$ has cardinality $2-1=1$. Therefore the cardinality of $$\sum_{G\cdot C_t\in X \text{, $C_t$ nodal}}
\wt^G(C_t)$$
is 3, recovering the classical result that there are 3 nodal conics in a pencil spanned by two conics in general position. 

In general, a nodal orbit $G\cdot C_t$ might contain one or multiple conics with branches that are not fixed. Rather than taking the cardinality of $\sum \wt^G(C_t)$, we could also take the cardinality of the $H$-fixed points of $\sum \wt^G(C_t)$ for any subgroup of $H$ of $G$, and this is not guaranteed to be 3. Instead, this will give a signed count of nodal conics fixed by $H$. If $G=\op{Gal(\C/\R)}$ acts on $\CP^2$ by pointwise complex conjugation, $G$-fixed points of \eqref{eq:intro_main_eq} recover the weighted count of nodal conics in $X$ defined over $\mathbb{R}$ up to a sign, as non-equivariantly the number of real conics in a pencil is $-(|\Sigma(\mathbb{R})|-1)$ rather than $|\Sigma| -1$. This comparison to a real signed count is discussed at the end of in this section, culminating in Corollary \ref{cor:real_count}.

Finally, we mention that Definition \ref{weight} is topologically motivated beyond the geometric description in terms of orbits of branches. Let $C_t$ be a nodal conic with stabilizer $G_t$ and nodal point $p$, which will also be fixed by $G_t$, and let $\nu\colon \tilde{C_t}\to C_t$ be the normalization. Observe that $\nu^{-1}(p)$ is a set of two points whose orbit type is the same as the set of branches $B_t$ of $C_t$ since $\tilde{C_t}\setminus \{\nu^{-1}(p)\}\cong C_t\setminus\{p\}$ equivariantly. The difference in equivariant Euler characteristics 
\begin{equation}\label{eq:normalization}
\chi^{G_t}(\tilde{C_t})-\chi^{G_t}(C_t) = \chi^{G_t}(\nu^{-1}(p))-\chi^{G_t}(p) = [\nu^{-1}(p)] - \{*\}
\end{equation} 
is equal to $\wt^{G_t}(C_t)$ in $A(G_t)$. 

The formula relating the base locus with the weighted sum of nodal orbits is stated in the main theorem:

\begin{theorem}\label{MAIN}Let $G$ be a finite group not isomorphic to either $\mathbb{Z}/2\times\mathbb{Z}/2$, $A_4$, or $D_8$, and assume $G$ acts linearly on $\CP^2$. Let $X$ be a $G$-invariant general pencil of conics in $\CP^2$, and let $[\Sigma]$ in $A(G)$ represent the base locus of $X$. Then  
\begin{equation}\label{maineq}\sum_{ \{G\cdot C_t\colon C_t \in X \text{ is nodal}\}} \wt^G(C_t)=[\Sigma]-\{*\}\end{equation}
 in $A(G)$. That is, there is a weighted count of orbits of nodal conics in $X$, valued in the Burnside ring of $G$.
\end{theorem}

This will be proved directly by explicitly checking that the formula holds for all possible invariant pencils of conics. Examples \ref{example_S4} and \ref{exa:Z4} give non-trivial examples of invariant general pencils of conics for $S_4$ and $\Z/4$ respectively. First, we prove two lemmas. 

\begin{lemma}
The only finite groups that can act invariantly on a general pencil of plane conics in $\CP^2$ are subgroups of $S_4$. 
\end{lemma}
\begin{proof}  If $G$ is a finite group that acts invariantly on a pencil of conics, then $G$ must fix the base locus of the pencil, i.e., $G$ must act bijectively on a set of four distinct points.  Thus $G$ must be a subgroup of $S_4$. \end{proof}

Any finite group that acts linearly on $\CP^2$ must be a finite subgroup of $PGL(3,\mathbb{C})$, a reference for which can be found in \cite{HL88}. We ultimately want to prove the theorem for each finite group $G$ that can act linearly on $\CP^2$ and invariantly on a pencil of conics. Thus we only need to consider linear group actions of subgroups of $S_4$, which is indeed a finite subgroup of $PGL(3,\mathbb{C})$.

\begin{lemma}\label{lem:nodal_pts_are_invariant}
    Let $G$ be a subgroup of $S_4$. Every $G$-invariant set $\Sigma$ of 4 points in general position in $\CP^2$ uniquely determines a $G$-invariant general pencil of plane conics. The set of 3 nodal conics in the pencil is $G$-invariant, as is the set of nodal points of the nodal conics.  
\end{lemma}
\begin{proof} Given a $G$-invariant set $\Sigma$ of $4$ general points in $\CP^2$, the $G$-invariant general pencil of conics $$X = \{sf(x,y,z)+tg(x,y,z)=0\colon [s:t]\in \CP^1\}\subseteq \CP^2$$ with base locus $\Sigma$ has generators $f(x,y,z)$ and $g(x,y,z)$ given by the homogenous degree 2 polynomials that generate the kernel of the linear map $\C^6\to \C^4$ that evaluates $x^2, y^2, z^2, xy, xz,$ and $yz$ on the points of $\Sigma$. Write $\Sigma = \{a,b,c,d\}$ so that  $L_{ij}\subseteq \CP^2$ is the line through $i,j\in \Sigma$. The 3 nodal conics in $X$ are given by the 3  products of distinct line pairs through $\Sigma$, we denote them $C_1 = L_{ab}\cup L_{cd}$, $C_2 = L_{ac}\cup L_{bd}$, and $C_3 = L_{ad}\cup L_{bc}$. Since $\Sigma$ is $G$-invariant, so is the set of lines $\{L_{ab}, L_{ac}, L_{ad}, L_{bc}, L_{bd}, L_{cd}\}$, whence so is the set of nodal conics $\{C_1, C_2, C_3\}$. See Figure \ref{figure:lines_through_Sigma}.

Each nodal conic $C_i$ has 1 nodal point $p_i$. The  $p_i$ are the intersection of the lines defining the nodal conic $C_i$, e.g., $p_1 = L_{ab}\cap L_{cd}$. Any $g\in G$ acts transitively on the orbit $G\cdot C_1$, and 
\[
g\cdot p_1 = g\cdot (L_{ab}\cap L_{cd}) = gL_{ab} \cap gL_{cd} = L_{ga,gb}\cap L_{gc,gd}
\] 
 in $G\cdot C_1$. This is not special to $C_1$, and in particular $g\cdot p_i$ is the nodal point of $g\cdot C_i$ for all $g\in G$ and nodal conics $C_i$. Thus the set of nodal points of the nodal conics, $\{p_1, p_2, p_3\}$, is also a $G$-invariant subset of $\CP^2$. We conclude that a $G$-invariant set $\Sigma$ of 4 general points in $\CP^2$ uniquely determines a $G$-invariant general pencil of plane conics, and the sets of 3 nodal conics and their nodal points are invariant under the action on $\CP^2$. \end{proof}

In other words, given a $G$-invariant set of 4 general points $\Sigma$, the 3 partitions of $\Sigma$ into 2 subsets of 2 points defines a map $G\hookrightarrow S_4\twoheadrightarrow S_3$, which need not be injective or surjective to imply that $G$ acts invariantly on the set of 3 nodal conics and their nodal points. 
As in the non-equivariant case, that a $G$-invariant set of 4 points in general position determines a $G$-invariant general pencil of conics can also be seen by a vector space observation: Every $G$-set $[\Sigma]$ of four general points in $\CP^2$ uniquely determines a pencil of conics which is the unique $G$-invariant 1-dimensional subspace of $\mathbb{P}\text{Sym}^2((\mathbb{C}^3)^{\vee})$ vanishing on $\Sigma$. 

Conversely, it is also clear that given a $G$-invariant pencil of general conics in $\CP^2$, there is a uniquely determined $G$-set of four points in general condition corresponding to the base locus of the pencil. 

\noindent \textit{Proof of Theorem \ref{MAIN}.} If $H_1, H_2\leq G$ are conjugate subgroups of a finite group $G$, then $A(H_1)\cong A(H_2)$, a proof can be found in \cite{Bou00}.  Thus we will only check that the theorem is true for each conjugacy class of subgroups of $S_4$. The conjugacy classes are:
\[e=\langle()\rangle\hspace{10mm} \Z/2\cong \langle(12)\rangle\hspace{10mm}S_3\cong\langle(123),(12)\rangle\]
\[ A_3=\{(), (123),(132)\}\hspace{10mm}\mathbb{Z}/4\cong\langle (1234)\rangle\hspace{10mm}A_4=\langle(123),(12)(34)\rangle \]
\[\mathbb{Z}/2\times\mathbb{Z}/2\cong\langle(12)(34),(13)(24)\rangle\hspace{10mm} D_8\cong\langle(1234),(13)\rangle \hspace{10mm} S_4. \]

For each  of these groups except  $\mathbb{Z}/2\times\mathbb{Z}/2$, $A_4$, and $D_8$, one can directly show that the theorem is true by computing weights of orbits of lines through $[\Sigma]$ for every possible orbit decomposition of $\Sigma$. Note that there may not be a $G$-invariant set of 4 points in general position in $\CP^2$ for \textit{every} possible orbit configuration of 4 points fixed by $G$, though there does exist a non-trivial $G$-invariant general pencil of conics for every subgroup of $S_4$, including for the groups $\Z/2\times \Z/2, A_4$, and $D_8$ for which the theorem is not true. See Examples \ref{example_S4} and \ref{exa:Z4} for thorough examples of invariant pencils of conics for $S_4$ and $\Z/4$. 

Showing  for each subgroup $G$ of $S_4$ that equation (\ref{maineq}) holds for any possible configuration of $[\Sigma]\in A(G)$ will prove the theorem.  
 The details are shown for $\mathbb{Z}/2$, $S_3$, and $A_3$, the same method works verbatim for $\Z/4$ and $S_4$.  

We will write $[\Sigma]=\{b_1, b_2, b_3, b_4\}\in A(G)$ for the base locus of a pencil, and the line through any $b_i$ and $b_j$ will be denoted by $L_{ij}$. Any nodal conic through $[\Sigma]$ has irreducible components given by the union of a pair of lines $\{L_{ij},L_{kl}\}$, which will be denoted as a $G$-set by $[L_{ij},L_{kl}]$ in $A(G)$.  

If $G=e$ is the trivial group, then any group action on $\CP^2$ is trivial. Thus this is simply the classical result over $\mathbb{C}$. 

If $G=\mathbb{Z}/2\cong\langle(12)\rangle$,  
the only genuine size four $G$-sets, and therefore the only  possible choices for  $[\Sigma]$ in $A(G)$, are the following: 
\begin{enumerate}
\item $[\Sigma]=4\{*\}$
\item $[\Sigma]=2[G]$
\item $[\Sigma]=2\{*\}+[G]$
\end{enumerate}

The fact that  $[\Sigma]$  
must be one of these cases relies on the fact that any genuine $G$-set $[S]\in A(G)$ has the form $$[S]=\sum_{(H_i)\colon H_i\leq G}n_i[G/H_i]=n_0[G/G]+n_1[G/e]$$ with $n_0, n_1\in\mathbb{Z}_{\geq 0}$ being the number of orbits with stabilizer equal to $G$ or $e$ respectively. Since $[\Sigma]$ is a genuine $G$-set, it must have one of the three configurations listed above.

Given a configuration of $[\Sigma]$, if there is a $G$-invariant pencil of conics $X$ determined by $[\Sigma]$, then the set of irreducible components of any nodal conic in $X$ is determined by one of the three configurations of a pair of distinct lines through $[\Sigma]$. Thus to see that the theorem is true for every configuration of $[\Sigma]$, and therefore true for $G=\mathbb{Z}/2$, we will compute the weight of each orbit of lines through any configuration of $[\Sigma]$ and show that the sum of the weights is equal to $[\Sigma]-\{*\}$ in $A(G)$. 

First consider the case where $[\Sigma]=4\{*\}$. All four points of $[\Sigma]$ are fixed, so $$\stab([L_{12}, L_{34}])=\stab([L_{13}, L_{24}])=\stab([L_{14}, L_{23}])=G$$ and each branch is fixed. Hence $\wt^G([L_{ij}, L_{kl}])=\{[L_{ij}, L_{kl}]\}-\{*\}=2\{*\}-\{*\}=\{*\}$ for any all $i,j,k,l\in\{1,2,3,4\}$. Hence the left-hand side of equation (\ref{maineq}) is $\sum \wt^G(B_t)=3\{*\},$ and the right-hand side of equation (\ref{maineq}) is $[\Sigma]-\{*\}=4\{*\}-\{*\}=3\{*\}$. 
 
Consider the second case where $[\Sigma]=2[G]$, and say that $\{b_1, b_2\}$ and $\{b_3, b_4\}$ are the orbits of $[\Sigma]$. In this case, $e$ is the element that acts trivially and $(12)$ is the element that acts nontrivially on each orbit, i.e., swaps $b_1$ and $b_2$ and swaps $b_3$ and $b_4$. Then for $g\in G$,
\[ 
g\cdot \{L_{12}, L_{34}\}= \begin{cases} 
      L_{12}, L_{34}, & \text{for }g=e \\
      L_{21}, L_{43}, & \text{for }g=(12) \\
   \end{cases}
  \]
  \[g\cdot \{L_{13}, L_{24}\}=\begin{cases}
   L_{13}, L_{24},& \text{for }g=e\\
   L_{24}, L_{13},& \text{for }g=(12)\\ \end{cases}
   \]
   \[
   g\cdot \{L_{14}, L_{23}\}=\begin{cases}
   L_{14}, L_{23},&\text{for }g=e\\
   L_{23}, L_{14},&\text{for }g=(12)\\ \end{cases}.
\]

The stabilizer of each nodal orbit is $G$, and so $\wt^G([L_{ij}, L_{kl}])=[L_{ij}, L_{kl}]-\{*\}$. Note that $[L_{12}, L_{34}]=2\{*\}$ because the branches are fixed by $G$, but $[L_{13}, L_{24}]$ and $[L_{14}, L_{23}]$ are both equal to $[G]$ in $A(G)$ because the branches are swapped by $G$.  Hence 
$$\wt^G([L_{12}, L_{34}])=[L_{12}, L_{34}]-\{*\}=2\{*\}-\{*\}=\{*\},$$
$$\wt^G([L_{13}, L_{24}])=[L_{13}, L_{24}]-\{*\}=[G]-\{*\}, \text{ and } $$
$$\wt^G([L_{14}, L_{23}])=[L_{14}, L_{13}]-\{*\}=[G]-\{*\}.$$
 Thus the left-hand side of equation (\ref{maineq}) is $\{*\}+2[G]-2\{*\}=2[G]-\{*\}$, and the right-hand side of equation (\ref{maineq}) is $[\Sigma]-\{*\}=2[G]-\{*\}$, as desired.

The last configuration of $[\Sigma]$ is $2\{*\}+[G]$. Say that $b_1$ and $b_2$ are the fixed points and $\{b_3, b_4\}$ are an orbit with $(12)$ swapping $b_3$ and $b_4$. Thus 
\[ 
g\cdot \{L_{12}, L_{34}\}= \begin{cases} 
      L_{12}, L_{34},&\text{for }g=e \\
      L_{12}, L_{43} ,&\text{for }g=(12) \\
   \end{cases}
\]\[
   g\cdot \{L_{13}, L_{24}\}=\begin{cases}
   L_{13}, L_{24},&\text{for }g=e\\
   L_{14}, L_{23},&\text{for }g=(12)\\ \end{cases}
   \]
   \[
  g\cdot  \{L_{14}, L_{23}\}=\begin{cases}
   L_{14}, L_{23},&\text{for }g=e\\
   L_{13}, L_{24},&\text{for }g=(12)\\ \end{cases}.
\]

Here, $\stab([L_{12}, L_{34}])=G$ and both lines are fixed, so $$\wt^G([L_{12}, L_{34}])=[L_{12}, L_{34}]-\{*\}=2\{*\}-\{*\}=\{*\}.$$ Note that $\stab([L_{13}, L_{24}])=\stab([L_{14}, L_{23}])=e$. Also,  
$(12)\cdot\{L_{13}, L_{24}\}=\{L_{14}, L_{23}\}$ and $(12)\cdot\{L_{14}, L_{23}\}=\{L_{13}, L_{24}\}$, so they are both in the same orbit. 
 Therefore we only need to count one of $G\cdot\{L_{13}, L_{24}\}$ or $G\cdot\{L_{14}, L_{23}\}$ in the weighted sum of nodal curves in the pencil determined by $[\Sigma]$. Making an arbitrary choice and using Lemma \ref{inflation}, 
 \begin{align*}
 \wt^G([L_{13}, L_{24}])=\infl_{e}^G(\wt^{e}([L_{13}, L_{24}]))&=\infl_{e}^G(2\{*\}-\{*\})\\
 &=\infl_{e}^G(\{*\})\\
 &=[G/e]=[G].
 \end{align*}

Finally, the left-hand side of equation (\ref{maineq}) is $$\wt^G([L_{12}, L_{34}])+\wt^G([L_{13}, L_{24}])=\{*\}+[G]$$ and the right-hand side of equation (\ref{maineq}) is $$[\Sigma]-\{*\}=[G]+2\{*\}-\{*\}=[G]+\{*\},$$ as desired. Therefore the theorem is true for $\mathbb{Z}/2$.

If $G=S_3\cong\langle(123),(12)\rangle$, the only possibilities for $[\Sigma]$ in $A(G)$ are: 
\begin{enumerate}
\item $[\Sigma]=4\{*\}$
\item $[\Sigma]=\{*\}+[G/\langle (12)\rangle]$
\item $[\Sigma]=2\{*\}+[G/\langle(123)\rangle]$
\end{enumerate}

The first case has been covered before, and is the same as the $\mathbb{Z}/2$ case when $[\Sigma]=4\{*\}$. 

Consider the second case where $[\Sigma]=\{*\}+[G/\langle (12)\rangle]$. Say $b_4$ is fixed and $\{b_1, b_2, b_3\}$ are an orbit so $\{b_1, b_2, b_3\}=[G/\langle(12)\rangle]=\{[()],[(123)],[(132)]\}$ in $A(G)$.  
Using the same method as for $\mathbb{Z}/2$ to find the stabilizer and orbit of each node, we observe 
that $\langle(12)\rangle$ is the stabilizer of all three sets of branches through $[\Sigma]$. Furthermore, all nodes are in the same orbit because $(123)\cdot \{L_{12}, L_{34}\}=\{L_{14}, L_{23}\}$, $(123)\cdot \{L_{14}, L_{23}\}=\{L_{13}, L_{24}\}$, and $(123)\cdot\{L_{13}, L_{24}\}=\{L_{12}, L_{34}\}$. 

Given that all nodes are in the same orbit, as in the third case for $\mathbb{Z}/2$ we only need to count one weighted node in the orbit to obtain the left-hand side of equation (\ref{maineq}). Arbitrarily choosing $[L_{12}, L_{34}]$, the branches of $[L_{12}, L_{34}]$ are equal to $2\{*\}$ in $A(\langle(12)\rangle)$. Thus $\wt^G([L_{12}, L_{34}])=\infl_{\langle(12)\rangle}^G(2\{*\}-\{*\})=[G/\langle(12)\rangle]$ 
Therefore, the left-hand side of equation (\ref{maineq}) is $[G/\langle (12)\rangle]$ and the right-hand side of equation (\ref{maineq}) is $[\Sigma]-\{*\}= [G/\langle(12)\rangle]$, as desired.

The last case to consider for $S_3$ is when $[\Sigma]=2\{*\}+[G/\langle(123)\rangle]$. Say that $b_3$ and $b_4$ are fixed and $\{b_1, b_2\}$ is an orbit with $b_1=[()]$ and $b_2=[(12)]$. Then $\stab([L_{12}, L_{34}])=G$ and $\wt^G([L_{12}, L_{34}])=2\{*\}-\{*\}=\{*\}$.

We can use the same method used for $\mathbb{Z}/2$ to find the stabilizer and orbit of each remaining node.  In this case only one of $[L_{13}, L_{24}]$ or $[L_{14}, L_{23}]$ needs to be counted in the left-hand side of equation (\ref{maineq}) because they are in the same orbit      with stabilizer $\langle(123)\rangle$. Arbitrarily choosing $[L_{13}, L_{24}]$, we see that $\wt^G([L_{13}, L_{24}])=\infl_{\langle (123)\rangle}^G(\{*\})=[G/\langle(123)\rangle]$.     Therefore, the left-hand side of equation (\ref{maineq}) is $[G/\langle(123)\rangle]+\{*\}$ and the right-hand side of equation (\ref{maineq}) is $[\Sigma]-\{*\}=2\{*\}+[G/\langle(123)\rangle]-\{*\}=[G/\langle(123)\rangle]+\{*\}$, as desired. Therefore the theorem is true for $S_3$.

   If $G=A_3=\{(), (123),(132)\}$, the only possibilities for $[\Sigma]$ in $A(G)$ are: 

\begin{enumerate}
\item $[\Sigma]=4\{*\}$
\item $[\Sigma]=\{*\}+[G]$
\end{enumerate}
Both cases can be checked using similar methods as $G=S_3$, no new ideas appear for $A_3$. The same is true for $G=\mathbb{Z}/4$ and $G=S_4$; one checks using the same ideas that \eqref{maineq} holds for all orbit configurations of base loci of an invariant pencil. \hfill$\Box$

We now give explicit examples of $G$-invariant pencils of conics for $S_4$ and $Z/4$. 

\begin{example}\label{example_S4} 
Let $V_{triv}$ denote the trivial 1-dimensional $G=S_4$ representation and $V_{sign}$  the 1-dimensional sign representation. Define the 3-dimensional $S_4$ representation $$V= V_{triv}\oplus V_{triv}\oplus V_{sign}$$ and consider the induced action on $\mathbb{P}V\cong \CP^2$. Let 
\[
\Sigma = \{[1:0:0], [0:1:0], [1:1:2], [1:1:-2]\}. 
\]
$\Sigma$ is an $S_4$-invariant set of points in $\CP^2$, and no three points of $\Sigma$ are colinear. Further, $[\Sigma]= 2\{*\}+[G/A_4]$ in $A(S_4)$ since $V_{sign}$ is trivial when restricted to $A_4$ and $[1:0:0], [0:1:1]$ are fixed points. Thus $[\Sigma]$ defines a general, $S_4$-invariant pencil of conics in $\CP^2$ 
\[
X = \{sf(x,y,z)+tg(x,y,z)=0\colon [s:t]\in \CP^1\}\subseteq \CP^2
\]
with defining equations 
\[
f(x,y,z) = -z^2+xy, \] \[g(x,y,z) = xz-yz.
\]
Labeling the points of $\Sigma$ as $a=[1:0:0],$ $b= [0:1:0],$ $c=[1:1:2],$ and $d=[1:1:-2]$, equations for the nodal conics in $X$ are: 
\begin{itemize} 
\item $L_{ab}\times L_{cd} = z(4x-4y)$, which is nodal at $[1:1:0]$. 
\item $L_{ac}\times L_{bd} = (2y-z)(2x+z)$, which is nodal at $[1:-1:-2]$. 
\item $L_{ad}\times L_{bc} = (2y+z)(2x-z)$, which is nodal at $[1:-1:2]$. 
\end{itemize}
where $L_{ij}$ denotes the equation for the line through points $i$ and $j$ in $\Sigma$. Writing $N$ for the set of nodal points,
\[
N= \{[1:1:0], [1:-1:-2], [1:-1:2]\}, 
\]
we see $N$ is an $S_4$-invariant set of 3 points in $\CP^2$ and $[N]=\{*\} + [S_4/A_4]$ in the Burnside ring $A(S_4)$.

One checks directly that $\wt^G([L_{ab},L_{cd}])=\{*\}$,  as the stabilizer of $L_{ab}\times L_{cd}$ is $G$ and the branches are fixed. $L_{ac}\times L_{bd}$ and $L_{ad}\times L_{bc}$ are in the same orbit, so we need only count the weight of one to count the whole orbit. Arbitrarily choosing $L_{ac}\times L_{bd}$, we see $\wt^G([L_{ac}, L_{bd}])=[G/A_4]$. Thus 
\begin{align*}
\sum_{ \{G\cdot C_t\colon C_t \in X \text{ is nodal}\}} \wt^G(C_t) &= \wt^G([L_{ab},L_{cd}]) + \wt^G([L_{ac},L_{bd}])  \\
&= \{*\} + [G/A_4]\\
&= [\Sigma]-\{*\}
\end{align*}
since $[\Sigma] = 2\{*\}+[G/A_4]$, as desired.
\end{example}

For any subgroup $H$ of $S_4$, we can restrict the action of $S_4$ above to $H$ to obtain an example of an $H$-invariant pencil of conics for which Theorem \ref{MAIN} holds. This constructs examples of Theorem \ref{MAIN} for $\Z/2, S_3, A_3$, and $\Z/4$. We do another example for $\Z/4$ with a different action on $\CP^2$:

\begin{example}\label{exa:Z4}
Let $V$ denote the standard 3-dimensional $S_4$ representation. $V$ is also is a $G=\Z/4$ representation obtained by restricting the action on $V$ to $\Z/4\leq S_4$, and $\Z/4$ acts linearly on $\mathbb{P}V\cong \CP^2$ with this action. Let 
\[
\Sigma = \{[3:2:1], [1:-2:-1], [1:2:-1], [1:2:3]\}. 
\]
$\Sigma$ is an $\Z/4$-invariant set of points in $\CP^2$. One can check that no three points of $\Sigma$ are colinear, and $[\Sigma]= [\Z/4]$ in $A(\Z/4)$. Thus $[\Sigma]$ defines a general, $\Z/4$-invariant pencil of conics in $\CP^2$ 
\[
X = \{sf(x,y,z)+tg(x,y,z)=0\colon [s:t]\in \CP^1\}\subseteq \CP^2
\]
with defining equations 
\[
f(x,y,z) = -x^2 + y^2 -z^2+2xz, \] \[g(x,y,z) = -2x^2+y^2-2z^2+xy+yz.
\]
Labeling the points of $\Sigma$ as $a=[3:2:1],$ $b= [1:-2:-1],$ $c=[1:2:-1],$ and $d=[1:2:3]$, equations for the nodal conics in $X$ are: 
\begin{itemize} 
\item $L_{ab}\times L_{cd} = (y-2z)(2x-y)$, which is nodal at $[1:2:1]$. 
\item $L_{ac}\times L_{bd} = (x-y-z)(x+y-z)$, which is nodal at $[1:0:1]$. 
\item $L_{ad}\times L_{bc} = (x-2y+z)(x+z)$, which is nodal at $[-1:0:1]$. 
\end{itemize}
where $L_{ij}$ denotes the equation for the line through points $i$ and $j$ in $\Sigma$. Writing $N$ for the set of nodal points,
\[
N = \{[1:2:1], [1:0:1], [-1:0:1]\}, 
\]
we see $N$ is $\Z/4$-invariant in $\CP^2$ and $[N]=[G/\langle(13)(24)\rangle] + \{*\}$ in the Burnside ring $A(\Z/4)$, where $[1:2:1]$ and $[-1:0:1]$ are in an orbit of size 2, stabilized by $\langle(13)(24)\rangle$, and $[1:0:1]$ is a $\Z/4$-fixed point. 

One checks directly that the stabilizer of $L_{ac}\times L_{bd}$ is all of $\Z/4$, and $\wt^G([L_{ac},L_{bd}])=[G/\langle(13)(24)\rangle]-\{*\}$. Further, $L_{ab}\times L_{cd}$ and $L_{ad}\times L_{bc}$ are in the same orbit, so we need only count the weight of one to count the whole orbit. Observe $\wt^G([L_{ab},L_{cd}])=[G]-[G/\langle(13)(24)\rangle]$.  Thus 
\begin{align*}
\sum_{ \{G\cdot C_t\colon C_t \in X \text{ is nodal}\}} \wt^G(C_t) &=  \wt^G([L_{ab},L_{cd}]) + \wt^G([L_{ac},L_{bd}]) \\
&= [G]-[G/\langle(13)(24)\rangle] + [G/\langle(13)(24)\rangle]-\{*\}\\
&= [G]-\{*\} \\
&= [\Sigma]-\{*\}
\end{align*}
since $[\Sigma] = [G]$, as desired.
\end{example}

Restricting the action of $\Z/4$ in the example above to any subgroup $H$ of $\Z/4$ gives an example of an $H$-invariant general pencil of conics for which Theorem \ref{MAIN} holds, so Example \ref{exa:Z4} also gives examples for $\Z/2, S_3$, and $A_3$. 

\begin{remark}\label{rmk:general_assumption}
    The assumption in Theorem \ref{MAIN} that $\Sigma$ is both $G$-invariant and consists of 4 points in general position is crucial. Given a $G$-invariant set of 4 points that are not in general position, e.g., 4 points on a line, one cannot  expect to move the points into general position without breaking $G$-invariance. 
\end{remark}

There is a  connection between Theorem \ref{MAIN} and the real signed count of nodal conics in a pencil. Non-equivariantly, the count of nodal conics in an general pencil is 3. Of these 3 nodal conics, the number of real conics is not invariant in general; different choices of 4 points may define different pencils of conics with different counts of real nodal conics. However, there is a certain signed count of real conics which is invariant, where real conics are designated a sign of $-1$ or $+1$ according to whether they are split or non-split. A plane curve $C$ with nodal point $p$ is called \textit{split} at $p$ if $p$ is a real point and $C$ has two real branches at $p$, \textit{non-split} at $p$ if $p$ is a real point and $C$ has two complex conjugate branches at $p$, and is complex if $p$ is not a real point. For example, $x^2-y^2=(x-y)(x+y)$ locally defines a split node, $x^2+y^2 = (x+iy)(x-iy)$ locally defines a non-split node. 

Given any set of 4 points in general position in $\CP^2$, we can write $n$ for the number of real points and $m$ for the number of complex conjugate pairs of points so that $4=n+2m$. Weighting the real split nodes in a general pencil of conics by $-1$ and the real non-split nodes by $+1$, the signed count of real nodal conics in any pencil with base locus consisting of $n$ real points and $m$ complex conjugate pairs of points is invariant. In other words, the real signed count does not depend on the actual choice of $n$ real and $m$ complex conjugate pairs of points, only that there are $n$ and $m$-many points in the base locus. Writing $W_n$ for the signed count of real nodal conics in a pencil $X$ with base locus $\Sigma$ consisting of $n = |\Sigma(\R)|$ real points, it is a consequence of \cite[Example 11.15]{L20} that 
\[
W_n = -(|\Sigma(\R)|-1).
\]

Now consider $G=\op{Gal}(\C/\R)\cong \Z/2$ acting on $\CP^2$. Let $\Sigma$ be a $G$-invariant set of $4$ general points in $\CP^2$, and say 
\[
[\Sigma] = n\{*\}+m[G]
\]
in $A(G)$, so again $\Sigma$ is $G$-invariant and consists of $n$ real points and $m$ complex conjugate pairs of points. $\Sigma$ defines a $G$-invariant general pencil of plane conics, $X$. 

Let $C_1, C_2,$ and $C_3$ denote the 3 nodal conics in $X$. Note that if $C_i$ is real, whether split or non-split, then $\stab(C_i)=G$. The branches of the split node are equal to $2\{*\}$ as a $G$-set since each branch is real and fixed by $G$. The branches of the non-split node are equal to $[G]$ as a $G$-set since they are a complex conjugate pair of 2 lines, swapped by the group action. This can also be seen from the normalization comparison $\wt^G(C_t) = \chi^G(\tilde{C_t})-\chi^G(C_t)$ in  \eqref{eq:normalization}. We  deduce 
\begin{equation}\label{eq:weight_signs}
\wt^G(C_i)=
\begin{cases} 
      2\{*\}-\{*\} = \{*\}, & \text{if $C_i$ is real, split} \\
      [G]-\{*\}, & \text{if $C_i$ is real, non-split} \\
   \end{cases}
\end{equation}
in $A(G)$. The cardinality of $\wt^G(C_i)$ is 1 regardless of whether $C_i$ is split or non-split, and the fixed point cardinalities of the weights of split and non-split nodes in \eqref{eq:weight_signs} are
\begin{equation}\label{eq:weight_fixed_points}
|\wt^G(C_i)^G|=
\begin{cases} 
      1, & \text{if $C_i$ is real, split} \\
      -1, & \text{if $C_i$ is real, non-split} \\
   \end{cases}
\end{equation}
In particular, it is immediate that $-|\wt^G(C_i)^G|$ recovers the non-equivariant sign of a real split or non-split node. Note finally that $|\Sigma^G| = |\Sigma(\R)| = n$. We have shown
\begin{corollary}\label{cor:real_count}
    Let $G=\op{Gal}(\C/\R)$ act on $\CP^2$ and let $[\Sigma]=n\{*\}+m[G]$ be a $G$-invariant set of 4 points in general position in $\CP^2$. Let $X$ be the $G$-invariant general pencil of conics with base locus $\Sigma$. Then
    \begin{equation}\label{eq:real_count_eq}
    \sum_{ \{G\cdot C_t\colon C_t \in X \text{ is nodal}\}} -|\wt^G(C_t)^G|=-(|\Sigma(\R)|-1)
    \end{equation}
    is equal to the signed count of real nodal conics in $X$. 
\end{corollary}

\section{Counterexamples} \label{counterexamplesection}

This section will give counterexamples where equation (\ref{maineq}) does not hold, which is for groups isomorphic to $\mathbb{Z}/2\times\mathbb{Z}/2$, $A_4$, and $D_8$.

\begin{counterexample}\label{counterexample_K4} First consider the case when $$G=\mathbb{Z}/2\times\mathbb{Z}/2=\{e,(12)(34),(13)(24),(14)(23)\}.$$ We have an action of $S_4$, and therefore of $G$, on $\CP^2$ using the standard $PGL(3,\mathbb{C})$-representation of $S_4$  
given by 
$$g_1:=e\mapsto\begin{bmatrix} 1 & 0 & 0 \\ 0 & 1 & 0 \\ 0 & 0 & 1 \\ \end{bmatrix}\hspace{5mm} \text{,}\hspace{5mm} g_2:=(12)(34)\mapsto\begin{bmatrix} -1 & 1 & 0 \\ 0 & 1 & 0\\ 0 & 1 & -1\\ \end{bmatrix},$$ 
$$g_3:=(13)(24)\mapsto\begin{bmatrix} 0 & -1 & 1 \\ 0 & -1 & 0 \\ 1 & -1 & 0 \\ \end{bmatrix}\hspace{5mm} \text{, and}\hspace{5mm} g_4:=(14)(23)\mapsto\begin{bmatrix} 0 & 0 & -1 \\ 0 & -1 & 0\\ -1 & 0 & 0 \\ \end{bmatrix}.$$
\bigbreak

Consider the point $p=[1,2,3]\in\CP^2$. Using the $g_i$ above to also denote the action on $\CP^2$, define the $G$-set 
\begin{align*}
[\Sigma]&=\{b_1:=g_1\cdot p, b_2:=g_2\cdot p, b_3:=g_3\cdot p, b_4:=b_4\cdot p\}\\
&=\{[1:2:3],[1:2:-1],[1:-2:-1], [-3:-2:-1]\}.
\end{align*}
We will show that $[\Sigma]$ is the $G$-invariant base locus of a general pencil, i.e., that no three points in $[\Sigma]$ are colinear, but that (\ref{maineq}) does not hold for the pencil of conics associated to $[\Sigma]$.

Note first that $[\Sigma]$ is $G$-invariant by construction, with $g\cdot b_i=b_{i+1}$ for $1\leq i \leq 3$ and $g\cdot b_4=b_1$ for all $g\neq e$ in $G$. Furthermore, $[\Sigma]$ was defined to be isomorphic to $G$ as a $G$-set, with the isomorphism being given by $b_i\mapsto g_i$ for $1\leq i \leq 4$. Therefore $[\Sigma]=[G]$ in $A(G)$. It is straightforward to check that no three points in $[\Sigma]$ lie on a line, so $\Sigma$ defines a $G$-invariant general pencil of conics.

Now we will show that equation (\ref{maineq}) does not hold for $[\Sigma]$. By observing where each element of $G$ maps each line through $\Sigma$, we can see that  each pair of lines through $[\Sigma]$ has stabilizer $G$. Each node has branches equal to $[G/H]$ for $H\leq G$ the subgroup that fixes both branches in addition to the union. Thus one can check that $$\wt^G([L_{12}, L_{34}])=[G/\langle(12)(34)\rangle]-\{*\},$$ $$\wt^G([L_{13}, L_{24}])=[G/\langle(13)(24)\rangle]-\{*\}, \text{ and}$$  $$\wt^G([L_{14}, L_{23}])=[G/\langle(14)(23)\rangle]-\{*\}.$$ 
Therefore the left-hand side of (\ref{maineq}) is 
\begin{equation}\label{eq:LHS_K4}[G/\langle(12)(34)\rangle]+[G/\langle(13)(24)\rangle]+[G/\langle(14)(23)]-3\{*\}.\end{equation}
The right-hand side of (\ref{maineq}) is $[\Sigma]-\{*\}=[G]-\{*\}$.

We will use Proposition 1.2.2 in \cite{TD79} to show the left-hand and right-hand sides of equation (\ref{maineq}) are not equal in $A(G)$.  In particular, we need to compute for each $K\leq G$ the number of $K$-fixed points of the left-hand side and the right-hand side of (\ref{maineq}). Writing $$S_1=[G/\langle(12)(34)\rangle]+[G/\langle(13)(24)\rangle]+[G/\langle(14)(23)]-3\{*\}$$ and $$S_2=[\Sigma]-\{*\}=[G]-\{*\}$$ for the left and right-hand sides of (\ref{maineq}), cardinalities of fixed points of subgroups of $G$ are: 
  \begin{center} 
\begin{tabular}{c|c|c}
$(K)\leq G$ & $|(S_1)^K|$ & $|(S_2)^K|$\\
\hline 
$e$  & 3 & 3 \\ 
$\langle(12)(34)\rangle$ & -2 & -1 \\ 
$\langle(13)(24)\rangle$ & -2 & -1 \\
$\langle(14)(23)\rangle$ & -2 &  -1\\ 
$G$ & -3 & -1\\
\end{tabular}
\end{center} 
\centerline{Table 2: Fixed point cardinalities of Equation \eqref{maineq} for subgroups of $(\Z/2)^2$}

The fact that there are subgroups of $G$ for which the number of fixed points of the LHS and RHS are not equal implies that the two sets are not equal in $A(G)$. Therefore equation (\ref{maineq}) fails for $G=\mathbb{Z}/2\times\mathbb{Z}/2$ and $[\Sigma]=[G]$. 
\end{counterexample}

\begin{counterexample}\label{counterexample_A4} Let $V$ be the standard 3-dimensional representation of $S_4$, and consider the induced action on $\mathbb{P}V\cong \CP^2$. $A_4$ acts on $\CP^2$ by restricting the action of $S_4$ on $\CP^2$. Again let 
\[
\Sigma = \{[1:2:3], [1:2:-1], [1:-2:-1], [3:2:1]\}. 
\]
$\Sigma$ is an $A_4$-invariant set of points in $\CP^2$. Furthermore, $[\Sigma] = [A_4/A_3]$ in the Burnside ring $A(A_4)$. The points of $\Sigma$ are in general position, and a basis for the $A_4$-invariant pencil of conics with base locus $\Sigma$ is given by 
\[
X=\{sf(x,y,z)+tg(x,y,z)=0\colon [s:t]\in \CP^1\}\subseteq \CP^2
\]
with 
\[
f(x,y,z) = -x^2+y^2-z^2 +2xz, \] \[g(x,y,z) = -2x^2 +y^2 -2z^2 + 2xy + 2yz.
\]
Labeling the points of $\Sigma$ as $a=[1:2:3],$ $b= [1:2:-1],$ $c=[1:-2:-1],$ and $d=[3:2:1]$, equations for the nodal conics in $X$ are: 
\begin{itemize} 
\item $L_{ab}\times L_{cd} = (2x-y)(y-2z)$, which is nodal at $[1:2:1]$. 
\item $L_{ac}\times L_{bd} = (x+y-z)(x-y-z)$, which is nodal at $[1:0:1]$. 
\item $L_{ad}\times L_{bc} = (x-2y+z)(x+z)$, which is nodal at $[-1:0:1]$. 
\end{itemize}
where $L_{ij}$ denotes the equation for the line through points $i$ and $j$ in $\Sigma$. The set of nodal points of the 3 nodal conics, $N = \{[1:2:1],[1:0:1],[-1:0:1]\}$, is $A_4$-invariant and equal to $[A_4/K_4]$ in $A(A_4)$. 

All three nodal conics are in the same orbit. The stabilizer of $L_{ab}\times L_{cd}$ is $K_4=\Z/2\times \Z/2$, with branches equal to $[K_4/\langle(14)(23)\rangle]$ in $A(K_4)$. Thus $$\wt^{A_4}([L_{ab},L_{cd}])=\op{inf}_{K_4}^{A_4}([K_4/\langle(14)(23)\rangle]-\{*\}) = [A_4/\langle(14)(23)\rangle] - [A_4/K_4]$$ in $A(A_4)$. Since all of $L_{ab}\times L_{cd}$, $L_{ac}\times L_{bd}$, and $L_{ad}\times L_{bc}$ are all in the same orbit, the left-hand side of \eqref{maineq} is 
\[
\sum_{ \{A_4\cdot C_t\colon C_t \in X \text{ is nodal}\}} \wt^{A_4}(C_t) = \wt^{A_4}([L_{ab},L_{cd}]) = [A_4/\langle(14)(23)\rangle] - [A_4/K_4].
\]
The right-hand side of \eqref{maineq} is 
\[
[\Sigma]-\{*\} = [A_4/A_3]-\{*\}.
\]
Since the cardinality of the $H$-fixed points of $[A_4/\langle(14)(23)\rangle] - [A_4/K_4]$ and $[A_4/A_3]-\{*\}$ are different for each nontrivial $H\le A_4$, 
\[
\sum_{ \{A_4\cdot C_t\colon C_t \in X \text{ is nodal}\}} \wt^{A_4}(C_t) \neq [\Sigma]-\{*\}
\]
in $A(A_4)$ by the same argument in Counterexample \ref{counterexample_K4} using \cite[Proposition 1.2.2]{TD79}. 
\end{counterexample} 

Theorem \ref{MAIN} is true for the other configurations of base loci for $A_4$,  $4\{*\}$ and $\{*\} + [A_4/(\mathbb{Z}/2)^2]$, and can be checked directly in the same manner as the proof of Theorem \ref{MAIN}. Counterexample \ref{counterexample_A4} is interesting because $\Sigma = \{[1:2:3], [1:2:-1], [1:-2:-1], [3:2:1]\}$ is an $H$-invariant set of $4$ points for any subgroup $H$ of $A_4$, and Theorem \ref{MAIN} holds for the $H$-invariant pencil of conics with base locus $\Sigma$ for $H$ which are strict subgroups of $A_4$. For example, the same base locus defines a $\Z/4$-invariant pencil of conics for which Theorem \ref{MAIN} is true, see Example \ref{exa:Z4}.

\begin{counterexample}\label{counterexample_D8}Finally, we'll construct a counterexample for $G=D_8$ 
using a different approach. We will start with a 3-dimensional representation of $D_8$ on $(\mathbb{C}^3)^{\vee}$ to obtain a 6-dimensional representation of $D_8$ on $V:=\text{Sym}^2((\mathbb{C}^3)^{\vee})$. The $G$-invariant vector space $V$ has a decomposition into irreducible sub-representations using the common eigenspaces of the generators of $D_8$, and from these irreducible sub-representations the pencils of conics correspond to the spans of irreducible 1-dimensional sub-representations. 

Write $r:=(13)$ and $s:=(1234)$ for the generators of  $D_8$. For reference, the character table of $D_8$ is given below, where $\chi_1,\chi_2, \chi_3$, and $\chi_4$ are the four 1-dimensional representations of $D_8$ and $\sigma$ is the unique 2-dimensional representation of $D_8$. The character of any 3-dimensional representation of $D_8$ is  given by $\chi=\sigma+\chi_i$ or $\chi=\chi_i+\chi_j+\chi_k$, $i,j,k\in\{1,2,3,4\}$. We will produce two counterexamples to Theorem \ref{MAIN} using a 3-dimensional representation of $W$ with character $\chi=\sigma+\chi_i$. 

\begin{center} 
\begin{tabular}{c c|c c|c c}
\text{ } & $e$ & $r^2$ & $r$ & $s$ & $sr$  \\
\hline
$\chi_1$ &1  &1  &1 &1 &1  \\ 
$\chi_2$ & 1 & 1 &1 & -1 & -1 \\ 
$\chi_3$ & 1 & 1 &-1 &1 &-1 \\
$\chi_4$ & 1 & 1 & -1 & -1 & 1 \\
$\sigma$ & 2 &  -2 & 0& 0&0 \\
\end{tabular}
\end{center}
\centerline{Table 3: Character table of $D_8$}

The unique 2-dimensional representation of $D_8$ is given by 
\begin{displaymath}
r\mapsto\begin{bmatrix} 0 & -1\\ 1 & 0\end{bmatrix}, \hspace{10mm} s\mapsto\begin{bmatrix} 1 & 0 \\ 0 & -1\end{bmatrix}.
\end{displaymath}
Therefore a 3-dimensional $D_8$ representation of $(\mathbb{C}^3)^\vee$ with basis $\{x,y,z\}$ and with character $\sigma+\chi_i$ is given by
\begin{displaymath}
r\mapsto \begin{bmatrix} 0 & -1 & 0\\1 & 0 & 0 \\ 0 & 0 & a\end{bmatrix}=:M_r, \hspace{10mm} s\mapsto\begin{bmatrix} 1 & 0 & 0 \\ 0 & -1 & 0 \\ 0 & 0 & b\end{bmatrix}=:M_s
\end{displaymath}

where $a,b\in\{\pm 1\}$ are equal to the values of $\op{tr}\chi_i(r)$ and $\op{tr}\chi_i(s)$ respectively. Using the basis $\{x^2,y^2,z^2,yz,xz,xy\}$ for $V$, observe that the 6-dimensional representation of $V$ obtained from the symmetric power of $W$ is given by 
  
$$r\mapsto \begin{bmatrix}0 & 1 & 0 & 0 & 0& 0\\ 1 & 0 & 0 & 0 & 0 & 0 \\0 & 0 & 1 & 0 & 0 & 0 \\ 0 & 0 & 0 & 0  & a & 0 \\ 0 & 0 & 0 & -a & 0 & 0 \\ 0 & 0 & 0 & 0 & 0 & -1\end{bmatrix}\text{, \hspace{3mm}} 
s\mapsto\begin{bmatrix}1 & 0 & 0 & 0 & 0 & 0\\ 0 & 1 & 0 & 0 & 0 & 0\\ 0 & 0 & 1 & 0 & 0 & 0\\ 0 & 0 & 0 & -b & 0 & 0 \\ 0 &0&0&0&b&0\\ 0&0&0&0&0&-1  \end{bmatrix}.
$$
  
Note $\text{Sym}^2(M_r)$ is the image of $r$ and $\text{Sym}^2(M_s)$ is the image of $s$. 
The common 1-dimensional $G$-invariant eigenspaces of $\text{Sym}^2(M_r)$ and $\text{Sym}^2(M_s)$ are $$z^2\text{, } xy x^2-y^2\text{, } \text{and } x^2+y^2.$$ 
There is also a 2-dimensional common $G$-eigenspace with basis $\{yz,xz\}$. Therefore, the possible $G$-invariant pencils of conics in $\mathbb{P}^2$ with action coming from the representation of $D_8$ on $V$ with character $\text{Sym}^2(\sigma+\chi_i)$ are: 
\begin{enumerate}
\item $\{\mu YZ+\lambda XZ=0\colon [\mu,\lambda]\in \CP^1\}$
\item $\{\mu Z^2+\lambda (X^2-Y^2)=0\colon [\mu,\lambda]\in \CP^1\}$
\item $\{\mu Z^2+\lambda(X^2+Y^2)=0\colon [\mu,\lambda]\in \CP^1\}$
\item $\{\mu Z^2+\lambda XY=0\colon [\mu,\lambda]\in \CP^1\}$
\item $\{\mu(X^2-Y^2)+\lambda(X^2+Y^2)=0\colon [\mu,\lambda]\in \CP^1\}$
\item $\{\mu(X^2-Y^2)+\lambda XY=0\colon [\mu,\lambda]\in \CP^1\}$
\item $\{\mu(X^2+Y^2)+\lambda XY=0\colon [\mu,\lambda]\in \CP^1\}$
\item $\{\mu(X^2-Y^2)+\lambda(a(X^2+Y^2)+bZ^2)=0\colon [\mu,\lambda]\in \CP^1\}$
\item $\{\mu XY+\lambda(a(X^2+Y^2)+bZ^2)=0\colon [\mu,\lambda]\in \CP^1\}$
\end{enumerate}
In the first 7 cases, one can check that the conics defining the pencil are not in general position. 
We will show that Theorem \ref{MAIN} doesn't hold for (8) above, and case (9) is similar.

 In the 8$^{\text{th}}$ case, $[\Sigma]= \{b_1, b_2, b_3, b_4\}$ where $$b_1=\left[1:1:i\sqrt{\frac{2a}{b}}\right]\text{, } b_2= \left[1:-1:i\sqrt{\frac{2a}{b}}\right],$$ $$b_3=\left[1:1:-i\sqrt{\frac{2a}{b}}\right] \text{, and } b_4= \left[1:-1:-i\sqrt{\frac{2a}{b}}\right].$$ Since the representation on $V$ is the symmetric power of the representation on $W$ given by $r\mapsto M_r$ and $s\mapsto M_s$, with $M_r$ and $M_s$ depending on the values of $a=\op{tr}\chi_i(r)$ and $b=\op{tr}\chi_i(s)$ respectively, the four cases we need to consider are $a=b=1$, $a=1$ and $b=-1$, $a=-1$ and $b=1$, and $a=b=-1$.

We will look at the case when $a=b=1$, as the others are similar. In this case, using the matrices $M_r$ and $M_s$ one can check that: 
\begin{center}
\begin{tabular}{ c c  }
$ e\cdot b_1=b_1$ & $(14)(23)\cdot b_1=b_1 $\\ 
$(13)\cdot b_1=b_2$ & $(1432)\cdot b_1=b_2$  \\  
$(13)(24)\cdot b_1=b_3$& $(12)(34)\cdot b_1=b_3$\\
 $(1234)\cdot b_1=b_4$ & $ (24)\cdot b_1=b_4$
\end{tabular} $\hspace{2mm}$
\end{center}
so that $[\Sigma]=[G/\langle(14)(23)\rangle]=\{b_1=[()], b_2=[(13)], b_3=[(13)(24)], b_4=[(24)]\}$. 
 
Using the method in the proof of Theorem \ref{MAIN} for finding stabilizers and orbits of each node, $$\stab([L_{12}, L_{34}])=\stab([L_{14}, L_{23}])=\{e,(13)(24),(13),(24)\}:= H_1,$$ and $H_1\cong \mathbb{Z}/2\times\mathbb{Z}/2$. Furthermore, $L_{12}, L_{34}$ and $L_{14}, L_{23}$ are in the same orbit in $A(G)$.  
Therefore we only need to count one of $[L_{12}, L_{34}]$ or $[L_{14} , L_{23}]$ in the left-hand side of equation (\ref{maineq}). Arbitrarily choosing $[L_{12}, L_{34}]$, observe that $[L_{12}, L_{34}]=[H_1/\langle(13)\rangle]$ in $A(H_1)$. Therefore, 
\begin{align*}
\wt^G([L_{12}, L_{34}])&=\inf\text{}_{H_1}^G([H_1/\langle(13)\rangle]-\{*\})\\
&=\left[\frac{H_1}{\langle(13)\rangle}\right]\cdot\left[\frac{G}{H_1}\right]-\left[G/H_1\right]\\
&=[G/\langle(13)\rangle]-[G/H_1]
\end{align*}
in $A(G)$. We can also observe that $\stab([L_{13}, L_{24}])=G$, and $[L_{13}, L_{24}]=[G/H_2]$ in $A(G)$ where $H_2:=\{e,(12)(34),(13)(24),(14)(23)\}$. Thus $$\wt^G([L_{13}, L_{24}])=[G/H_2]-\{*\}.$$

It is worth noting that $H_1\cong H_2$ in $S_4$, but $H_1$ and $H_2$ are \textit{not} conjugate in $D_8$. Therefore the two $G$-sets $[G/H_1]$ and $[G/H_2]$ are not equal in $A(G)$. The left-hand side of equation (\ref{maineq}) is $$\wt^G([L_{12}, L_{34}])+\wt^G([L_{13}, L_{24}])=[G/\langle(13)\rangle]-[G/H_1]+[G/H_2]-\{*\}.$$ 

Given that $[\Sigma]=[G/\langle(14)(23)\rangle]$, the right-hand side of equation (\ref{maineq}) is $[G/\langle(14)(23)\rangle]-\{*\}$.

As with $\mathbb{Z}/2\times\mathbb{Z}/2$, we can use \cite{TD79} Proposition 1.2.2  to show that Theorem \ref{MAIN} is not true for this case. In particular, we can show that for some $K\leq G$, the number of $K$-fixed points of the left and right-hand sides of are not equal. Writing $$S_1=[G/\langle(13)\rangle]-[G/H_1]+[G/H_2]-\{*\}$$ and $$S_2=[G/\langle(14)(23)\rangle]-\{*\}$$ for the left and right-hand sides of (\ref{maineq}), one sees 
\[
|(S_1)^{\langle(14)(23)\rangle}| = -1 \ne 3 = |(S_2)^{\langle(14)(23)\rangle}|. 
\]

The fact that the number of $\langle(14)(23)\rangle$-fixed points differ 
implies that the left-hand side and right-hand side of \eqref{maineq} are not equal in $A(G)$ in this case. Therefore Theorem \ref{MAIN} is not true for $D_8$. 

It is worth noting that even if $[G/H_1]=[G/H_2]$ in $A(G)$, Theorem \ref{MAIN} still would not hold. In that case, the left-hand side of equation (\ref{maineq}) would be $[G/\langle(13)\rangle]-\{*\}$ and the right-hand side would be $[\Sigma]-\{*\}=[G/\langle(14)(23)\rangle]-\{*\}$. The same issue arises, $[D_8/\langle(13)\rangle]=[D_8/\langle(14)(23)\rangle]$ in $A(S_4)$ because $\langle(13)\rangle$ and $\langle(14)(23)\rangle$ are conjugate in $S_4$, but not in $D_8$.  
The fact that  $D_8$ has subgroups which are conjugate in $S_4$ but not in $D_8$ is the crux of why Theorem \ref{MAIN} fails in for this group. \end{counterexample}

\newpage

\bibliographystyle{amsalpha}
\bibliography{250430_conics}

\end{document}